\documentclass[a4paper,legno,11pt]{article}

\usepackage[utf8]{inputenc}
\usepackage[T1]{fontenc}
\usepackage{xspace,
	amsthm,
	amsmath,
	amssymb,
	bbm,
	tikz,
	graphicx,
	subfig,
	keyval}
\DeclareGraphicsExtensions{.png, .pdf}

\usepackage[hidelinks]{hyperref}

\usepackage{geometry}
\tikzstyle{nodino}=[circle,draw,fill,inner sep=0pt,minimum size=0.5mm]
\tikzstyle{infinito}=[circle,inner sep=0pt,minimum size=0mm]
\tikzstyle{nodo}=[circle,draw,fill,inner sep=0pt, minimum size=0.5*width("k")]
\tikzstyle{nodo_vuoto}=[circle,draw,inner sep=0pt, minimum size=0.5*width("k")]
\usetikzlibrary{graphs}
\tikzset{every loop/.style={min distance=10mm,in=300,out=240,looseness=10}}
\tikzset{place/.style={circle,thick,draw=blue!75,fill=blue!20,minimum
		size=6mm}}
\tikzset{place2/.style={circle,thick,draw=red!75,fill=red!20,minimum
		size=6mm}}

\newcommand{\R}{\mathbb R}

\newcommand{\f}{\frac}
\newcommand{\rr}{{\mathbb R}}

\newcommand{\zz}{{\mathbb Z}}

\newcommand{\G}{{\mathcal{G}}}
\newcommand{\udot}{\|u'\|_{L^2(\mathcal{G})}}
\newcommand{\uLp}{\|u\|_{L^p(\mathcal{G})}}
\newcommand{\uLtwo}{\|u\|_{L^2(\mathcal{G})}}
\newcommand{\HmuG}{H_\mu^1(\mathcal{G})}

\newcommand{\elevel}{\mathcal E}

\newcommand\vv{\textsc{v}}

\newcommand{\dx}{\,dx}

\theoremstyle{plain} 
\newtheorem{thm}{Theorem}[section]

\newtheorem{prop}[thm]{Proposition} 

\theoremstyle{definition}

\theoremstyle{definition}

\theoremstyle{remark} 

\title{Quantum graphs and dimensional crossover: the honeycomb}
\author{Riccardo Adami, Simone Dovetta, Alice Ruighi}
\begin{document}
\maketitle	

\begin{abstract}
We summarize features and results on the problem of the existence of Ground States for the Nonlinear Schr\"odinger Equation on doubly-periodic metric graphs. We extend the results known for the two--dimensional square grid graph to the honeycomb, made of infinitely-many identical hexagons. Specifically, we show how the coexistence between one--dimensional and two--dimensional scales in the graph structure leads to the emergence of threshold phenomena known as dimensional crossover.

\end{abstract}

	\section{Introduction}
In the last decade there has been a dramatic increase in the study of the dynamics of systems on metric graphs, or {\em networks}. This is mainly due to two different issues: first, the extensive use of mathematics in topics traditionally confined to a more qualitative approach (e.g. biology, social sciences, economics); second, the flexibility and the simplicity of networks as a mathematical environment to model phenomena occurring in the actual world. 

Networks enter in the description of evolutionary phenomena on branched structures, namely, one-dimensional complexes made of {\em edges}, either finite or infinite, meeting at special points called {\em vertices}. Edges and vertices define the {\em topology} of the graph. The {\em metric} structure is defined by associating to every edge a {\em length} and then an arclength. This is easily accomplished by associating to every edge $e$ a coordinate
$x \in [0, \ell_e]$, where $\ell_e$ is the length of the edge.

Such a schema applies to signals propagating in networks, circuits, and to more recent scientific and technological challenges of the new emerging field of research called {\em Atomtronics}.

The first appearance of metric graphs in the mathematical modeling of natural systems dates back to 1953 and is due to Ruedenberg and Scherr \cite{RS53}, who modeled a naphtalene array as a network of edges and vertices arranged in a hexagonal lattice, like a honeycomb. Then, a Hamiltonian operator representing the quantum energy of the system was defined on such a structure, and its spectrum was computed in order to deduce the possible values of the energy of the valence electrons. The paper is not only a milestone in physical chemistry, but it also introduces some important mathematical tools like the so-called {\em Kirchhoff's conditions} at the vertices of the graph, and it opens the research field of quantum graphs. Dealing with a standard quanto-mechanical system, the model is governed by a linear equation, i.e. the Schr\"odinger equation of the system.

Since then, the use of metric graphs has become widespread in the literature, exiting the realm of quantum mechanics and extending to electromagnetism, acoustics, and many others physically relevant contexts. However, most of the models were linear. The first systematic introduction to nonlinear dynamics on graphs was given by Ali Mehmeti \cite{alimehmeti} in a nowadays classical treatise published in 1994, but one had to wait around three decades to see the analysis of the dynamics of a specific nonlinear model, first given in \cite{acfn2011} and concerning the effect of the impact of a fast soliton of the Nonlinear Schr\"odinger Equation (NLSE) on the vertex of an infinite star-graph. After this result, the resarch on the NLSE on graphs underwent an important development, especially because of great technical advances on the study of the mathematical aspects of the nonlinear Schr\"odinger Equation (especially following the seminal papers by Keel and Tao \cite{kt} and by Kenig and Merle \cite{km}) from one side, and because of the rapid evolution of the technology of Bose-Eintein condensates (BEC) from the other, and in particular of the new accomplishments in the construction of {\em traps} of various shapes, to be used in BEC experiments.

In order to motivate the mathematical problem we are dealing with, let us be more specific on this point.
A Bose-Einstein condensate is a system of a large (from thousands to millions) number of identical bosons, usually magnetically and/or optically confined in a spatial region, called {\em trap}. As predicted by Bose \cite{bose} and Einstein \cite{einstein}, under a prescribed value of the temperature, called ''critical value'', the system collapses into a very peculiar and non-classical state, in which:

\begin{itemize}

\item Every particle acquires an individual wave function (which is in general not the case for many-body systems, that are given a collective wave function only).

\item The wave function is the same for all particles, and is called {\em wave function of the condensate}.

\item The wave function of the condensate  solves the following variational problem:
\begin{equation} \label{gp}
\min_{u \in H^1 (\Omega), \int |u|^2 = N} E_{GP}(u)
\end{equation}
where
\begin{itemize}
\item $E_{GP}$ is the {\em Gross-Pitaevskii energy (GP) functional}, namely
\begin{equation} \label{gpe}
E_{GP} (u) \ = \ \| \nabla u\|_{L^2(\Omega)}^2 + 8 \pi \alpha \| u \|_{L^4 (\Omega)}^4
\end{equation}
($\alpha$ is the scattering length of the two-body interaction between the particles in the condensate);

\item $\Omega$ is the trap where the condensate is confined;

\item $N$ is the number of particles in the condensate;

\item provided it exists, the minimum corresponds to a standing wave for the Gross-Pitaevskii {\em Nonlinear Schr\"odinger} Equation
$$ i \partial_t \psi (t,x) \ = \ - \Delta \psi (t,x) + 32 \pi \alpha | \psi (t,x) |^2 \psi (t,x).
$$

\end{itemize}
\end{itemize}

\noindent Then it becomes an important issue to solve the problem of minimizing the functional \eqref{gpe} under the constraint $\int_\G |u|^2 dx = \mu$ given in \eqref{gp}.
As one might expect, the result heavily depends on $\Omega$, not only for what concerns the actual shape of the minimizer, but also for the sake of its mere existence. It is indeed this last issue that has been mostly studied during the last years, and will be the subject of the present note.

\subsection{Existence of Ground States: Results}
From now on, we consider a metric graph $\mathcal G$ and the NLS energy functional defined as
\begin{equation}
\label{energy}
E (u, \mathcal G) \ = \ \frac 1 2\int_\G |u'|^2 dx - \f 1 p \int_\G |u|^p dx.
\end{equation}
The first term is called {\em kinetic term}, as it represents the kinetic energy associated to the system, while the second is the {\em nonlinear term}.

The main difference of \eqref{energy} with respect to the GP energy \eqref{gpe} is that in \eqref{energy} a more general nonlinearity power is considered instead of the only case $p=4$, but we restrict to the so-called focusing case, where the nonlinear term has a negative sign, and encodes the fact that the two-body interaction between the particles is attractive. 

Owing to the choice of the sign, it is clear that 
there is
a competition between the two terms: the kinetic term favours widespread signals, while the nonlinear term prevents the minimizers from dispersing too much. When a minimizer exists, it always realizes a compromise between the two terms and the two corresponding tendencies: spreading or squeezing.

We study the problem of minimizing the energy \eqref{energy} with the constraint of constant mass, namely
	\begin{equation}
		\label{mass}
		\uLtwo^2= \int_\G |u|^2 \, dx = \mu > 0.
	\end{equation}
We shall use the notation	
	\begin{equation}
		\label{inf}
		\elevel(\mu):=\inf_{u\in\HmuG}E(u,\G),
	\end{equation}
and introduce the 	ambient space
	\begin{equation}
		\HmuG:=\{\,u\in H^1(\G)\,:\,\uLtwo^2=\mu\,\}
	\end{equation}
We call {\em ground state at mass $\mu$} or, for short, {\em ground state}, every minimizer of \eqref{energy} among all functions sharing the same mass $\mu$.

First of all, it is well-known \cite{cazenave, lions,zakharov}, that in the case of the real line, and provided that $2 < p < 6$, the compromise between kinetic and nonlinear term that gives rise to a ground state is realized for every $\mu$ by the {\em soliton}
\begin{equation}
\label{soliton}
\phi_\mu (x) \ = \ \mu^\alpha \phi_1 (\mu^\beta x), \qquad \alpha:=\frac{2}{p-2}, \,\,\, \beta: =\frac{p-2}{6-p}
,
\end{equation}
where the prototype soliton is denoted by $\phi_1$ and equals
$$ \phi_1 (x) : = C{\rm{sech}}(cx)\,.$$

\noindent In the case of a real half-line $\rr^+$, by elementary symmetry arguments one can immediately realize that a solution exists for every value of the mass $\mu$ and it coincides with a half-soliton with the maximum at the origin, possibly multiplied by a phase factor.

Despite the result for the half-line and for the line (i.e. a pair of half-lines), for the graph made of three half-lines meeting one another at a single vertex (i.e. a {\em star graph}) it has been proven that there is no ground state, irrespectively of the choice of $\mu$ (\cite{acfn12}). Starting from this negative result, the problem of ensuring (or excluding) the existence of ground states for the NLS on graphs gained some popularity in the community, and some general results were found, isolating a key topological condition (\cite{ast15}), studying in detail particular cases (\cite{DT-p,mp,nps}), dealing with compact graphs (\cite{CDS,dovetta}), introducing concentrated nonlinearities (\cite{DT,serratentarellly,st2,T}), focusing on the more challenging $L^2$-critical case (i.e. $p = 6$ \cite{ast17}). More recently, also some pioneering investigations of nonlinear Dirac equations has been initiated (\cite{BCT-p,BCT-19}).

\begin{figure}
	\centering
	\begin{tikzpicture}[xscale= 0.5,yscale=0.5]
	\draw[step=2,thin] (0,0) grid (6,6);
	\foreach \x in {0,2,...,6} \foreach \y in {0,2,...,6} \node at (\x,\y) [nodo] {};
	\foreach \x in {0,2,...,6}
	{\draw[dashed,thin] (\x,6.2)--(\x,7.2) (\x,-0.2)--(\x,-1.2) (-1.2,\x)--(-0.2,\x)  (6.2,\x)--(7.2,\x); }
	\end{tikzpicture}
	\caption{the two--dimensional square grid.}
	\label{grid}
\end{figure}
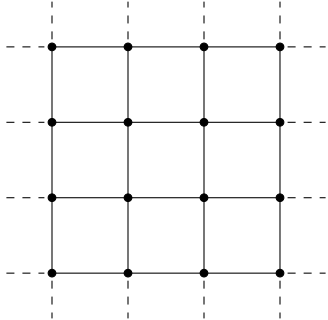

The analysis of NLS equations on periodic graphs has been developed for instance in (\cite{GPS,pankov,pelinovsky}), and a systematic discussion of the problem of ground states for periodic graphs has been carried out in \cite{dovetta-per}, however here we shall focus on a particular phenomenon highlited in \cite{adst} and called {\em dimensional crossover}. Investigating the problem of proving the existence or the nonexistence of ground states for the NLS on the regular two--dimensional square grid (see Figure \ref{grid}), it was found that three different regimes come into play:
\begin{enumerate}
\item if $2 < p < 4$, then a ground state exists for every value $\mu$ of the mass;
\item if $p > 6$, then there is no ground state
irrespectively of the value chosen for the mass;
\item if $p=6$, then there is a particular value of the mass, called {\em critical mass} and denoted by $\mu^*$, such that the infimum of the energy passes from $0$ to $-\infty$ as the mass exceeds $\mu^*$, and ground states never exist for any value of the mass;
\item if $4 \leq p < 6$, then there is a particular value of the mass, $\mu_p$, such that ground states exist only beyond $\mu_p$.
\end{enumerate}
Now, points 1 and 2 are common to what one finds in the problem of the ground states in $\R$ and $\R^2$. The transition of the actual value of the infimum of the energy as in Point 3 is characteristic of one-dimensional domains, in particular of quantum graphs made of a compact core and a certain number of half--lines. 

What really distinguishes the case of the grid graph from the previously studied cases of quantum graphs is point 4, where an unprecedented behaviour is detected for nonlinearity powers ranging from $4$ to $6$. Here power $4$ is meaningful since it is the critical power for two-dimensional problems! Then, the fact that power $4$ corresponds to a transition in the beaviour of the problem reveals that the two-dimensional structure is emerging.

Qualitatively, the grid is two-dimensional on a large scale, and this fact must emerge when searching for low-mass ground states, since low-mass means widespread functions.

From a quantitative point of view, the emergence of the two-dimensional large scale structure occurs in the validity of the {\em two-dimensional Sobolev inequality}, i.e.
\begin{equation}
\label{2sobolev} 
\| u \|_{L^2 (\G)} \ \leq \ C \| u' \|_{L^1 (\G)} \qquad(u\in W^{1,1}(\G)).
\end{equation}
As well-known in Functional Analysis, such an inequality is typical

 of two-dimensional domains, whereas in one-dimension one has the {\em one-dimensional Sobolev inequality}
\begin{equation}
\label{1sobolev} 
\| u \|_{L^\infty (\G)} \ \leq \ C \| u' \|_{L^1 (\G)}\qquad (u\in W^{1,1}(\G)).
\end{equation}
Now, inequality \eqref{1sobolev} is easy to prove for every one-dimensional non-compact graph, just using 
$$ u (x) = \int_\gamma u'(t) \, dt $$
where $x$ is any point of the graph and the symbol $\gamma$ denotes a path isomorphic to a half-line starting at $x$. The existence of such a path is ensured by the fact that the graph is non-compact (therefore it extends up to infinity) and connected (so that it is possible to reach the infinity from $x$ through a sequence of adjacent edges).

It is then clear that what marks the transition between the one and the two-dimensional regime is the coexistence of estimates \eqref{1sobolev} and \eqref{2sobolev}, so that what really characterizes the grid, as well as every structure dysplaying a two-dimensional nature in the large scale, is the validity of \eqref{2sobolev}.

As one shall expect, such a portrait can be generalized to the setting of periodic graphs exploiting higher dimensional structures in the large scale, like regular $n$--dimensional grids. In this context, it is readily seen that the dimensional crossover takes place between the one--dimensional and the $n$--dimensional critical power (see \cite{ad} for the explicit discussion of the case $n=3$).

\begin{figure}
	\centering
	\includegraphics[width=0.5\textwidth]{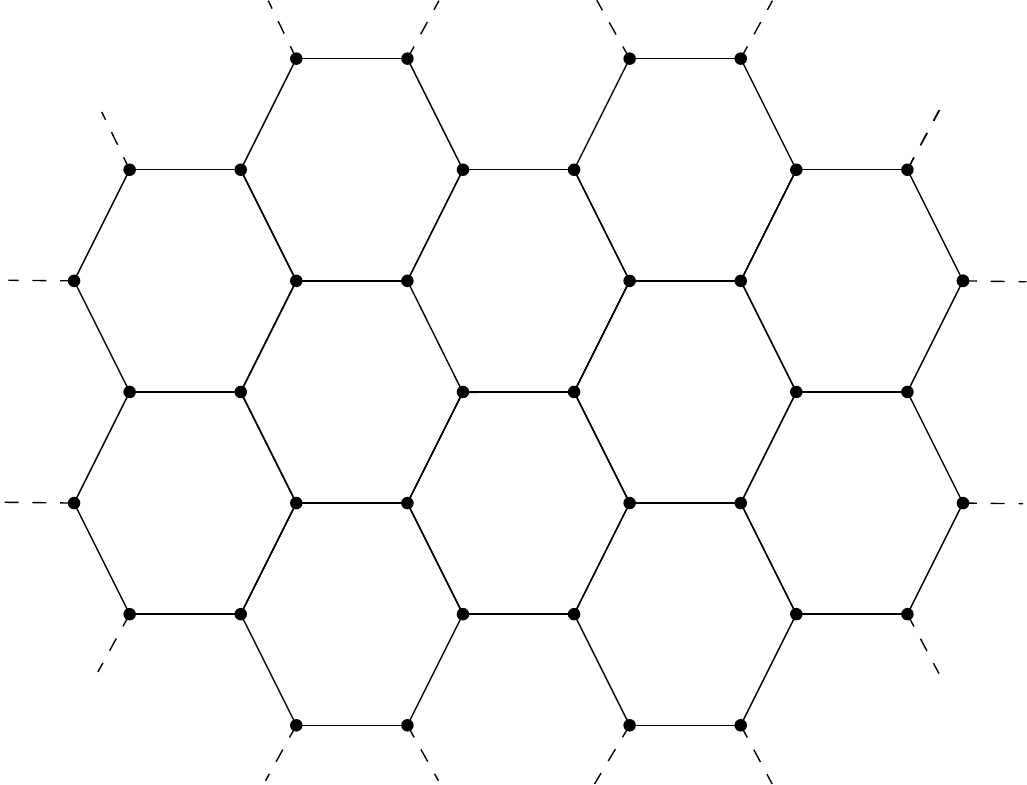}
	\caption{The infinite two-dimensional hexagonal grid $\G$.}
	\label{fig-grid}
\end{figure}

In this paper we show that for the {\em honeycomb graph}, namely the grid made of the periodic repetition of a hexagon along a two-dimensional mesh (see Figure \ref{fig-grid}), estimate \eqref{2sobolev} holds true. Moving from this fact, we deduce a complete result about the existence or nonexistence of ground states, closely following the steps intoduced in \cite{adst}.

\subsection{Existence of ground states in the honeycomb: the complete result}
According to the roadmap established in \cite{adst}, the validity of a Sobolev inequality 
results in the validity of a  corresponding family of Gagliardo-Nirenberg inequalities. Namely, from
\eqref{1sobolev} one obtains the {\em 1-dimensional Gagliardo-Nirenberg inequalities} that provide the following estimate of the potential term in \eqref{energy}:
\begin{equation}
\label{1gn}
\| u \|_{L^p (\G)}^p \ \leq \ C \| u' \|_{L^2(\G)}^{\f p 2 - 1} \| u \|_{L^2 (\G)}^{\f p 2 + 1},
\end{equation}
that, inserted in \eqref{energy}, gives
\begin{equation}
\label{subcritical}
E (u, \G) \ \geq \ \f 1 2 \| u' \|_{L^2 (\G)}^2 - \f C p \| u' \|_{L^2 (\G)}^{ \f p 2 - 1} \mu^{\f p 4 + \f 1 2}
\end{equation}
from which one immediately concludes that, if $2 < p < 6$, then
$$\mathcal{E}(\mu) > - \infty,$$
opening the possibility of existence of a ground state. In order to conclude for the existence, one should then consider the behaviour of minimizing sequences. By periodicity, the translation invariance of the problem excludes immediately escaping to infinity, so that the only possibility for a sequence not to converge is to spread along the grid, reaching in the limit zero energy. As a consequence, if there exists a function with negative energy, then minimizing sequences must converge and therefore a ground state exists.

The existence of a function with negative energy in the cases $2 < p <4 $ for every $\mu$, and $4 \leq p < 6$ for $\mu$ large enough, is the content of Theorem \ref{THM}  and of the positive part of point $(i)$ in Theorem \ref{THM 2}.

Conversely, to get to the core of our non-existence results, let us consider inequality \eqref{1gn} and notice that for $p=6$ it specializes to
\begin{equation}
\label{1gn6}
\| u \|_{L^6 (\G)}^6 \ \leq \ C \| u' \|_{L^2(\G)}^{2} \| u \|_{L^2 (\G)}^{4}.
\end{equation}
On the other hand, from \eqref{2sobolev} one derives
\begin{equation}
\label{2gn}
\| u \|_{L^p (\G)}^p \ \leq \ C \| u' \|_{L^2 (\G)}^{p - 2} \| u \|_{L^2 (\G)}^2,
\end{equation}
that,
for $p=4$, gives
\begin{equation}
\label{2gn4}
\| u \|_{L^4 (\G)}^4 \ \leq \ C \| u' \|_{L^2 (\G)}^{2} \| u \|_{L^2 (\G)}^2.
\end{equation}
Now, interpolating between \eqref{1gn6} and \eqref{2gn4} one has, for every $p \in [4,6]$
\begin{equation}
\label{interpolate}
\| u \|_{L^p (\G)}^p \ \leq \ C \| u' \| _{L^2 (\G)}^{2} \| u \|_{L^2 (\G)}^{p-2}.
\end{equation}
Then, by \eqref{interpolate}
\begin{equation}
\label{energybelow}
\begin{split}
E (u, \G) \ \geq \  & \f 1 2 \| u' \|_{L^2 (\G)}^2 - \f C p \| u' \|_{L^2 (\G)}^2 \| u \|_{L^2 (\G)}^{p-2} \\
\ = \ & \f 1 2 \| u' \|_{L^2 (\G)}^2 \left(1 - \f {2C} p \mu^{\f p 2 - 1} 
\right) \end{split}
\end{equation}
Then, for every $p \in [4,6]$ there exists a positive value $\mu_p > 0$ given by 

\[
\mu_p:=\Big(\f p {2C}\Big)^{\f 2 {p-2}}\,,
\]

\noindent with $C$ being the sharpest constant in \eqref{interpolate}, such that
\begin{itemize}
\item If $\mu < \mu_p$, then $E (u, \G) > 0$ for every $u \in H^1_\mu (\G)$. Since, by spreading the function $u$ along the grid, one immediately gets $\mathcal{E}(\mu) = 0$, it turns out that the infimum is not attained and ground states do not exist.

\item If $\mu > \mu_p$ it turns out that $\mathcal{E}(\mu) < 0$, and possibly infinitely negative.

\end{itemize}
The dimensional crossover lies exactly in this continuous transition from the subcritical regime (where for every mass there is a ground state) to the supercritical, where there are values of the mass in correspondence of which the energy is not lower bounded. In standard cases, such a transition only occurs in correspondence of the unique critical case, that amounts to $6$ in dimension one, and to $4$ in dimension two. In the case of a doubly periodic graph as the honeycomb we consider here, this actually takes place for all the nonlinearities $p$ between $4$ and $6$, so that a continuum of critical exponents arises between the critical power of dimension $2$ and the one of dimension $1$.

Here are the complete results:

	\begin{thm}
		\label{THM}
		Let $2<p<4$. Then, for every $\mu>0$, there exists a ground state of mass $\mu$.
	\end{thm}

	\begin{thm}
		\label{THM 2}
		For every $p\in[4,6]$ there exists a critical mass $\mu_p>0$ such that
		
		\begin{itemize}
			\item[\textit{(i)}] if $p\in(4,6)$ then ground states of mass $\mu$ exist if and only if $\mu\geq\mu_p$, and
			\begin{equation}
			\label{inf p46}
			\elevel(\G)\begin{cases}
			=0 & \text{if }\mu\leq\mu_p\\
			<0 & \text{if }\mu>\mu_p\,.
			\end{cases}
			\end{equation}
			\item[\textit{(ii)}] if $p=4$ then ground states of mass $\mu$ exist if $\mu>\mu_4$ and they do not exist if $\mu>\mu_4$. Furthermore, \eqref{inf p46} holds true also in the case $p=4$.
			\item[\textit{(iii)}] if $p=6$ then ground states never exist, independently of the value of $\mu$, and
			\begin{equation}
				\label{inf p6}
				\elevel(\mu)=\begin{cases}
				0 & \text{if }\mu\leq\mu_6\\
				-\infty & \text{if }\mu>\mu_6\,.
				\end{cases}
			\end{equation}
		\end{itemize}
	\end{thm}
	
\noindent Theorems 1.1 and 1.2 do not differ from their analogues in the case of the square grid, treated in \cite{adst}. The only remarkable new procedures concern the proof of Sobolev inequality as in Theorem \ref{thm-sob} and the construction of a function with negative energy proving the existence of a ground state in the regime $p\in(2,4)$.

\medskip
The remainder of the paper is organised as follows. Section \ref{sec:notation} sets some notation for the honeycomb, whereas Section \ref{sec:sobolev} develops the proof of Sobolev inequality \eqref{1sobolev}. Finally, within Section \ref{sec:competitor} we exhibit functions realizing strictly negative energy when $p\in(2,4)$, giving the proof of Theorem \ref{THM}.

	\section{Notation}
	\label{sec:notation}
	
	\begin{figure}
		\centering
		\subfloat[][]{
		\includegraphics[width=0.4\textwidth]{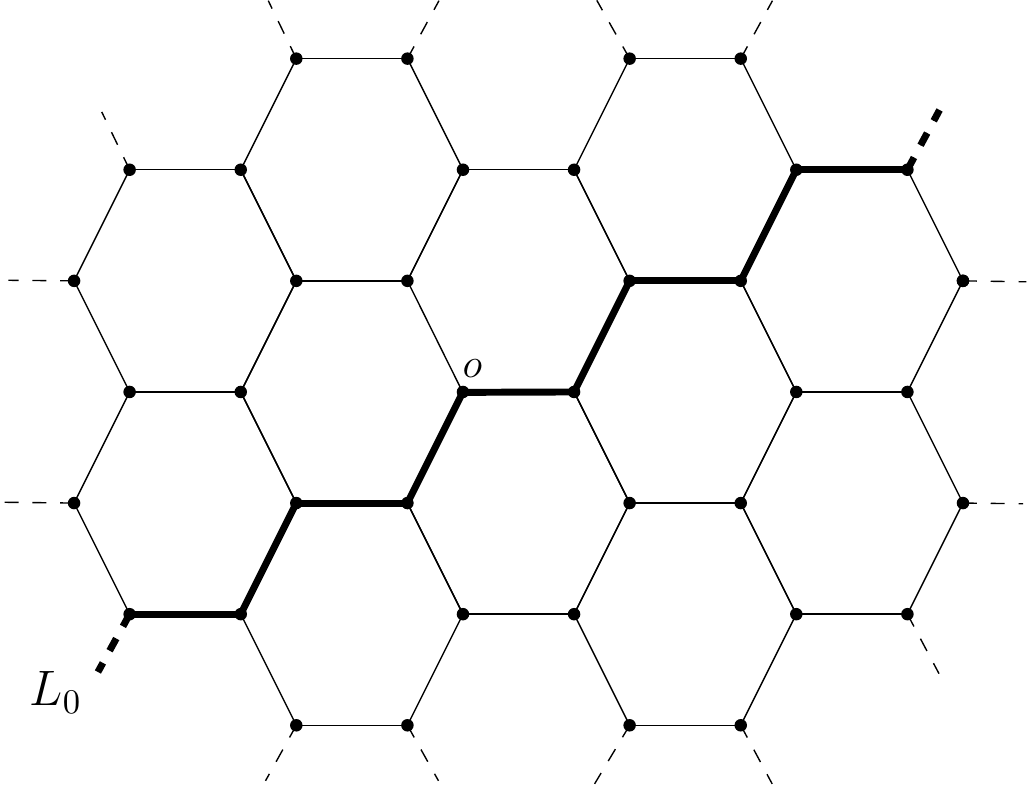}}\qquad
		\subfloat[][]{
		\includegraphics[width=0.4\textwidth]{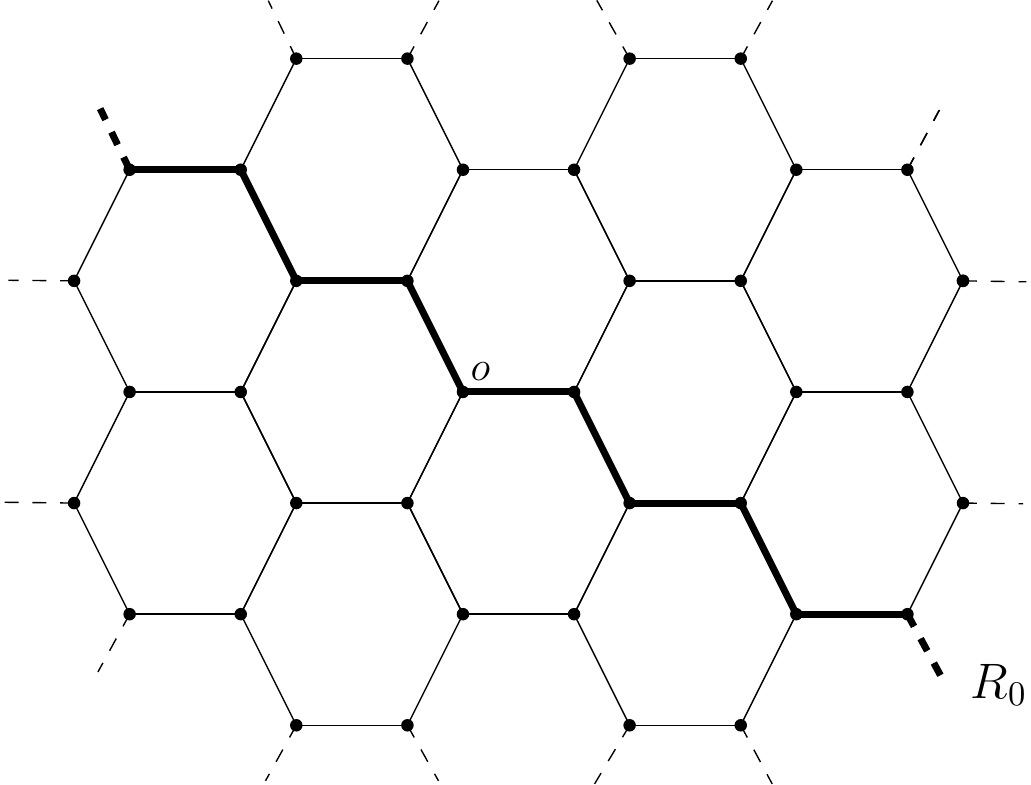}}
		\caption{the paths $L_0$ (a) and $R_0$ (b).}
		\label{fig-paths}
	\end{figure}

	Before going further, a bit of notation is necessary. Particularly, to ease several of the upcoming arguments, it is useful to decompose the exagonal grid in two family of parallel infinite paths, so that the whole graph $\G$ can be described as their union.
	
	To this purpose, let us introduce the following construction. Fix any cell in $\G$ and denote by $o$ its lower left vertex. Note that, starting at $o$, there is one horizontal edge on the right and both one up-directed and one down-directed edge on the left.  Consider then the infinite path running through $o$ built up this way. First, moving from $o$ to the right, follow the infinite path that alternates a horizontal and an up-directed edge. Then, moving from $o$ to the left, follow the infinite path that alternates a down-directed and a horizontal edge. We denote by $L_0$ the union of these two paths (see Figure \ref{fig-paths}(a)).
	
	Similarly, consider both the infinite path that goes from $o$ to the left alternating an up-directed and a horizontal edge, and the one that originates at $o$ and moves to the right alternating a horizontal and a down-directed edge. We denote the union of these two by $R_0$ (see Figure \ref{fig-paths}(b)).
	
	Note that both on $L_0$ and on $R_0$ natural coordinates $x_{L_0}:L_0\to(-\infty,+\infty),\,x_{R_0}:R_0\to(-\infty,+\infty)$ can be defined, so that they can be identified with real lines (with the origin corresponding to $o$). 
	
	Now, consider for instance the vertex belonging to $L_0$ which is at distance $2$ from $o$ on its right. It is immediate to see that an infinite path running through this vertex and parallel to $R_0$ can be recover simply repeating the procedure used to construct $R_0$. However, this is not the case if we consider the vertex of $L_0$ at distance $1$ from $o$ on its right, as it already belongs to $R_0$.
	
	More generally, through every vertex on $L_0$ located at an even distance from $o$ on its right runs an infinite path parallel to $R_0$. It is then straightforward to check that the same holds true also for every vertex on $L_0$ located at an odd distance from $o$ on its left (whereas vertices at even distances on the left do not provide any additional path). This leads to a family $\{R_j\}_{j\in\mathbb{Z}}$ of infinite parallel paths in $\G$.
	
	Analogously, one can consider the family of infinite paths $\{L_i\}_{i\in\mathbb{Z}}$ all parallel to $L_0$, which arises taking any vertex on $R_0$ either at an even distance from $o$ on its right or at odd distance from $o$ on its left and repeating the steps in the construction of $L_0$. 
	
	We stress the fact that the set defined by \( \Big(\bigcup_{i\in\mathbb{Z}}L_i\Big)\cap\Big(\bigcup_{j\in\mathbb{Z}}R_j\Big)\, \) is composed by all the horizontal edges of $\mathcal{G}$ and for this reason it follows
	
	\[
	\mathcal{G}\subset\Big(\bigcup_{i\in\mathbb{Z}}L_i\Big)\cup\Big(\bigcup_{j\in\mathbb{Z}}R_j\Big)\,.
	\]
	
	\noindent In particular $L_i\cap R_j\neq\emptyset$ for every $i,j\in\mathbb{Z}$, as they share exactly one horizontal edge. 
	
	Finally, given $i,j\in\mathbb{Z}$, we denote by $I_i ^j\subset L_i$  the union of the horizontal edge that $L_i$ shares with $R_j$ and the up-directed edge on its right. Moreover, we set $v_i ^j$ the first vertex of $I_i ^j$ that we meet walking down $R_j$ from $-\infty$ (see Figure \ref{fig-I}(a)). Note that, for every $i$, $L_i= \bigcup_{j\in\mathbb{Z}}I_i ^j$. Similarly, we define $J_j ^i$ as the union of the horizontal edge shared by $L_i$ and $R_j$ and the up-directed edge on its left. As before, we observe that, for every $j \in \mathbb{Z}$, $R_j=\bigcup_{i\in\mathbb{Z}}J_j ^i$ and again we denote by $w_j ^i$ the first vertex of $J_j ^i$ that we encounter walking down $L_i$ from $-\infty$ (Figure \ref{fig-I}(b)). 
	
	\begin{figure}
		\centering
		\subfloat[][]{
		\includegraphics[width=0.4\textwidth]{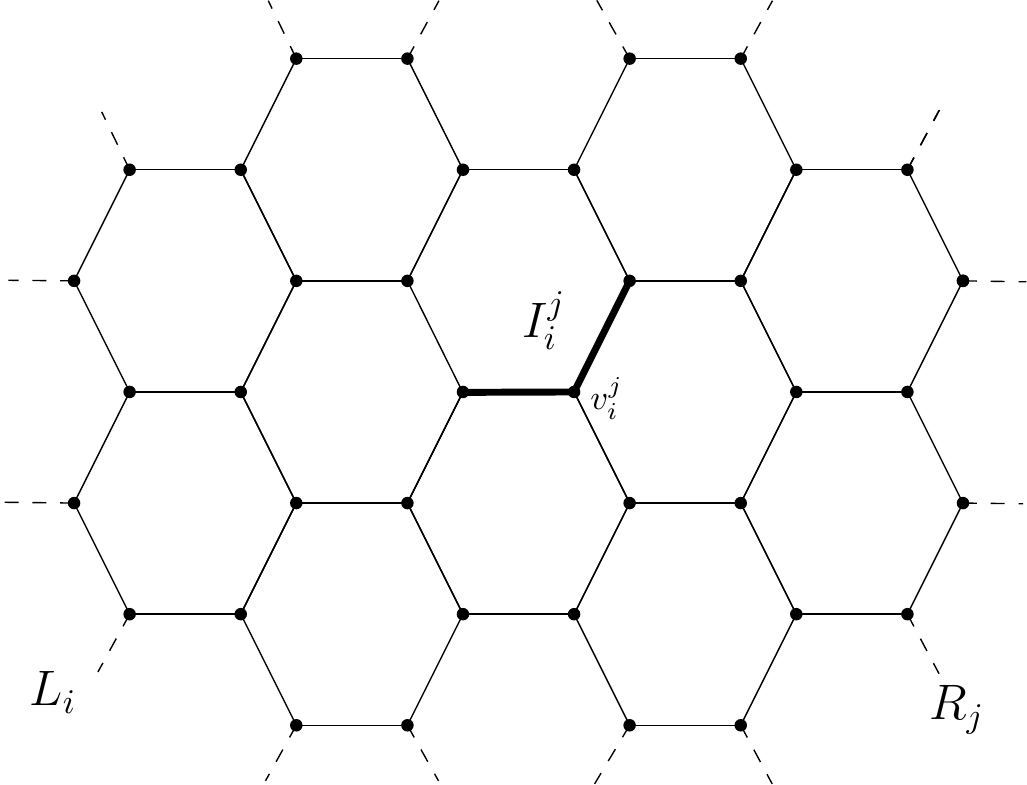}
		}\qquad
		\subfloat[][]{
		\includegraphics[width=0.4\textwidth]{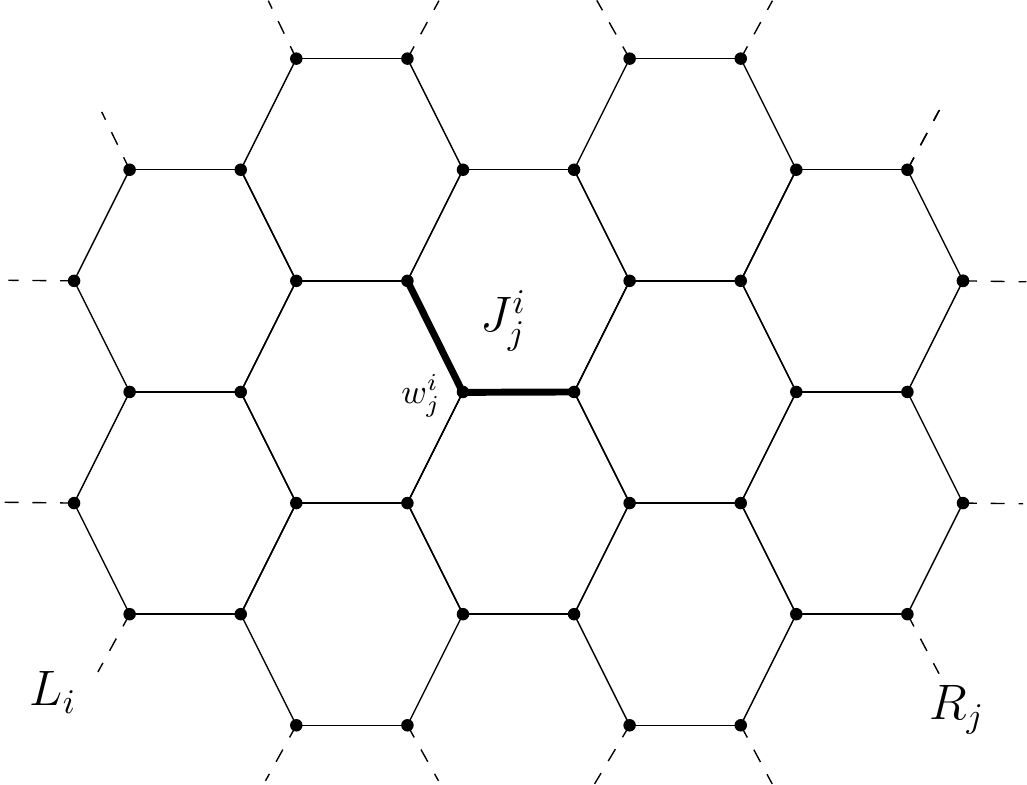}
		}
		\caption{The subsets $I_i^j$ (a) and $J_j^i$ (b).}
		\label{fig-I}
	\end{figure}

	\section{Sobolev inequality}
	\label{sec:sobolev}

This section is devoted to the derivation of some functional inequalities that are responsible for the grid $\G$ interpolating between one-dimensional and two-dimensional behaviours. Particularly, the two-dimensional nature of the graph shows up explicitly with the following result, stating the validity of the Sobolev inequality in the form typical of dimension two.
 	
\begin{thm} 
\label{thm-sob}
For every $u \in W^{1,1}(\mathcal{G})$,
\begin{equation}
\label{sobolev}
\|u\|_{L^2(\mathcal{G})} \leq 2\sqrt{2l} \|u'\|_{L^1(\mathcal{G})}.
\end{equation}
\end{thm}

\begin{proof}
We beforehand remind that $\mathcal{G}\subset\Big(\bigcup_{i\in\mathbb{Z}}L_i\Big)\cup\Big(\bigcup_{j\in\mathbb{Z}}R_j\Big)$, so that
\begin{equation}
\label{norm}
\|u\|^2 _{L^2(\mathcal{G})} \leq \sum_i \|u\|^2 _{L^2(L_i)} + \sum_j \|u\|^2 _{L^2(R_j)}. 
\end{equation} 
In order to prove (\ref{sobolev}), we aim to estimate the two terms on the right side of (\ref{norm}). Let us start with $\sum_i \|u\|^2 _{L^2(L_i)}$, where $\|u\|^2 _{L^2(L_i)} = \int_{L_i} |u(x)|^2 dx$.

Consider any point $x \in \mathcal{G}$ located on $L_i$. Observe that $x$ can be reached following at least two different paths on $\mathcal{G}$. The first one walks down $L_i$ from $-\infty$ to $x$, whereas the second one runs through $R_j$ from $-\infty$ to the vertex $v_i ^j$ and then moves on $L_i$ from $v_i ^j$ to $x$ (Figure \ref{fig-path}). Identifying with some abuse of notation the points $x$ and $v_i^j$ with their corresponding coordinates $x_{L_i}(x),\,x_{L_i}(v_i^j)$ and $x_{R_j}(v_i^j)$, we denote by $L_i(-\infty,x)$, $R_j(-\infty,v_i^j)$ and $L_i(v_i^j,x)$ the paths from $-\infty$ to $x$ along $L_i$, from $-\infty$ to $v_i^j$ along $R_j$ and from $v_i^j$ to $x$ along $L_i$, respectively.

Thus, we get 
\begin{equation}
\label{direct}
u(x)=\int_{L_i(-\infty,\,x)} u'(\tau) d\tau
\end{equation} 
and 
\begin{equation}
\label{indirect}
u(x)=\int_{R_j(-\infty,\,v_i ^j)} u'(\tau) d\tau + \int_{L_i(v_i ^j,\,x)} u'(\tau) d\tau.
\end{equation} 
Multiplying (\ref{direct}) and (\ref{indirect}) and using the fact that $L_i(-\infty,\,x)\subset L_i,\,R_j(-\infty,\,v_i^j)\subset R_j$ and $L_i(v_i^j,\,x)\subset I_i^j$, we estimate
\begin{align*}
|u(x)|^2 & = \bigg| \int_{L_i(-\infty,\,x)} u'(\tau) d\tau \bigg|\cdot \bigg| \int_{R_j(-\infty,\,v_i ^j)} u'(\tau) d\tau + \int_{L_i(v_i ^j,\,x)} u'(\tau) d\tau \bigg| \\ & \leq \bigg( \int_{L_i(-\infty,\,x)} |u'(\tau)| d\tau \bigg) \cdot \bigg( \int_{R_j(-\infty,\,v_i ^j)} |u'(\tau)| d\tau + \int_{L_i(v_i ^j,\,x)} |u'(\tau)| d\tau \bigg) \\ & \leq \bigg( \int_{L_i} |u'(\tau)| d\tau \bigg) \cdot \bigg( \int_{R_j} |u'(\tau)| d\tau + \int_{I_i ^j} |u'(\tau)| d\tau \bigg).
\end{align*}
Then, integrating over $L_i$ 
\begin{equation}
\label{stima}
\int_{L_i} |u(x)|^2 dx = \int_{L_i} |u'(\tau)| d\tau \bigg( \int_{L_i} \bigg( \int_{R_j} |u'(\tau)| d\tau + \int_{I_i ^j} |u'(\tau)| d\tau \bigg) dx \bigg).
\end{equation}
Recall that $L_i= \bigcup_{j \in \mathbb{Z}} I_i ^j$  and note that both $\int_{R_j} |u'(\tau)| d\tau$ and $\int_{I_i ^j} |u'(\tau)| d\tau$ are piecewise constant on each $I_i ^j$ as functions of $x$.  Hence, it results
\begin{equation}
\label{const1}
\int_{L_i} \bigg( \int_{R_j} |u'(\tau)| d\tau \bigg) dx = 2l \sum_{j \in \mathbb{Z}} \int_{R_j} |u'(\tau)| d\tau,
\end{equation}
and
\begin{equation}
\label{const2}
\int_{L_i} \bigg( \int_{I_i ^j} |u'(\tau)| d\tau \bigg) dx = 2l \sum_{j \in \mathbb{Z}} \int_{I_i ^j} |u'(\tau)| d\tau = 2l\int_{L_i} |u'(\tau)| d\tau.
\end{equation}
By (\ref{stima}), (\ref{const1}) and (\ref{const2}) it follows
\begin{align*}
\int_{L_i} |u(x)|^2 dx & = \int_{L_i} |u'(\tau)| d\tau \bigg( 2l \sum_{j \in \mathbb{Z}} \int_{R_j} |u'(\tau)| d\tau + 2l\int_{L_i} |u'(\tau)| d\tau \bigg) \\ & \leq 4l \| u' \| _{L^1(\mathcal{G})} \int_{L_i} |u'(\tau)| d\tau,
\end{align*}
as each term in the sum can be dominated by $\| u' \| _{L^1(\mathcal{G})}$.


\begin{figure}
	\centering
	\subfloat[][]{
		\includegraphics[width=0.4\textwidth]{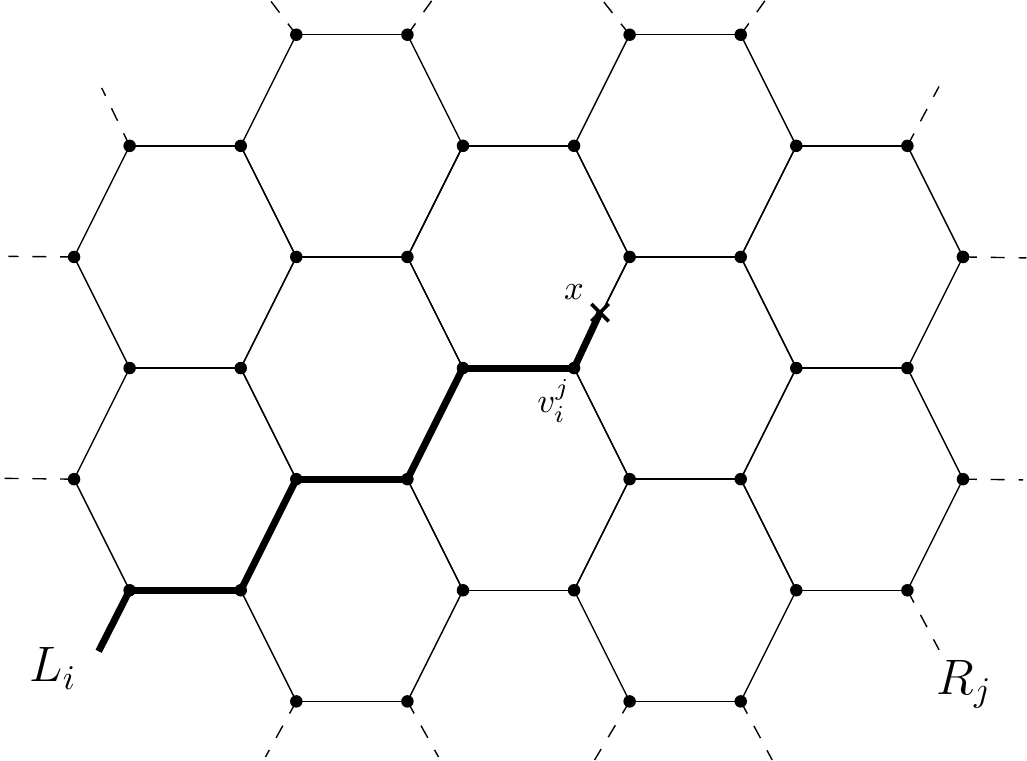}
	}\qquad
	\subfloat[][]{
		\includegraphics[width=0.4\textwidth]{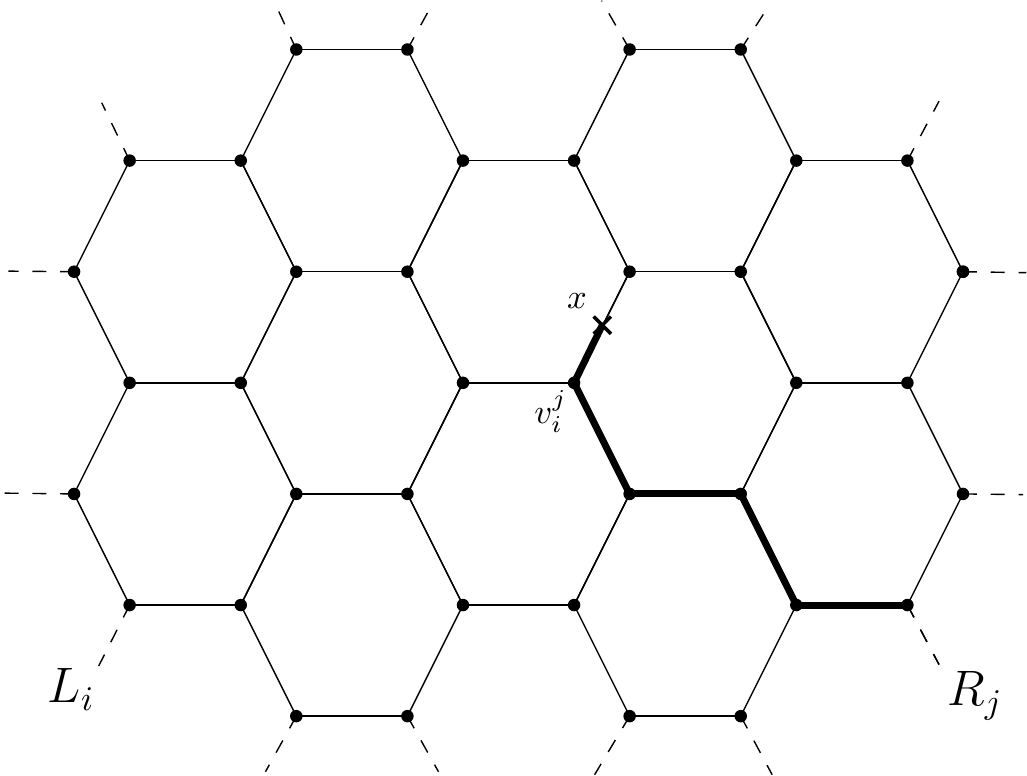}
	}
	\caption{The paths from $-\infty$ to $x$ along $L_i$ (a) and $R_j$ (b) as in the proof of Theorem \ref{thm-sob}.}
	\label{fig-path}
\end{figure}


Finally, summing over $i \in \mathbb{Z}$ yields at
\begin{equation*}
\sum_{i \in \mathbb{Z}} \int_{L_i} |u(x)|^2 \leq 4l\| u' \| _{L^1(\mathcal{G})} \sum_{i \in \mathbb{Z}} \int_{L_i} |u'(\tau)| d\tau \leq 4l \| u' \| ^2 _{L^1(\mathcal{G})}.
\end{equation*}
The same procedure can be adapted to estimate $\sum_{j \in \mathbb{Z}} \int_{R_j} |u(x)|^2 dx$, replacing $I_i^j$ with $J_j^i$ whenever need, so that by (\ref{norm}) we end up with
\begin{equation*}
\|u\|^2 _{L^2(\mathcal{G})} \leq 8l\| u' \| ^2 _{L^1(\mathcal{G})}.
\end{equation*}

\end{proof}

Arguing as in the proof of Theorem 2.3 in \cite{adst}, it can then be proved that Theorem \ref{thm-sob} entails the following two-dimensional Gagliardo-Nirenberg inequality on $\G$

\begin{equation}
	\label{2GN}
	\uLp^p\leq C \uLtwo^2\udot^{p-2}
\end{equation}

\noindent for every $u\in H^1(\G)$ and $p\geq2$ (here $C$ denotes a universal constant).

On the other hand, as for every non-compact metric graph, it is known that also the one-dimensional Gagliardo-Nirenberg inequality

\begin{equation}
	\label{1GN}
	\uLp^p\leq \uLtwo^{\frac{p}{2}+1}\udot^{\frac{p}{2}-1}
\end{equation}

\noindent holds true on $\G$, again for every $u\in H^1(\G)$ and $p\geq2$ (for a simple proof relying on the theory of rearrangements on graphs see for instance \cite{ast-jfa}).

Hence, combining \eqref{2GN}--\eqref{1GN}, a new version of the Gagliardo-Nirenberg inequality can be derived, which we refer to as {\em interpolated Gagliardo-Nirenberg inequality}, that accounts for the dimensional crossover in Theorem \ref{THM 2}. Indeed, for every $p\in[4,6]$ there exists a constant $K_p$, depending only on $p$, such that 

\[
\uLp^p\leq K_p\uLtwo^{p-2}\udot^2
\]

\noindent for every $u\in H^1(\G)$ (as the argument is the same, we refer to Corollary 2.4 in \cite{adst} for a complete proof of this fact).

	\section{Existence result: proof of Theorem \ref{THM}}
	\label{sec:competitor}
	
	Throughout this section, we provide the proof of Theorem \ref{THM}, showing that whenever $p$ is smaller than $4$, ground states always exist for every value of the mass. 
	
	To this purpose, we first recall a general compactness result, originally proved in Proposition 3.3 of \cite{adst}, which is valid for every doubly periodic metric graphs, so that it also applies in the case of the two-dimensional hexagonal grid we are dealing with.
	
	\begin{prop}[Proposition 3.3, \cite{adst}]
		\label{prop_comp}
		Let $p<6$ and $\mu<0$. If $\elevel(\mu)<0$, then a ground state with mass $\mu$ exists. 
	\end{prop}

	\begin{proof}[Proof of Theorem \ref{THM}]
		In view of Proposition \ref{prop_comp}, given $\mu>0$, it is enough to prove that $\elevel(\mu)<0$ to show that ground states in $\HmuG$ exist.
		
		We henceforth consider the following construction. For every $i\in\zz$, recall that $L_i$ is identified with a real line $(-\infty,+\infty)$ through a coordinate $x_{L_i}$, and we are free to choose which vertex $\vv\in L_i$ corresponds to the origin $x_{L_i}(\vv)=0$. We thus fix the origin of each $L_i$ in the following way. First, set the origin of $L_0$ at any of its vertices being the left endpoint of a horizontal edge. Then, since the up-directed edge on the left of this vertex connects $L_0$ with $L_1$, set the origin of $L_1$ at the endpoint of this bridging edge. Let then $\overline{L}_0$ be the straight line in the plane passing through both the origin of $L_0$ and the one of $L_1$. For each $i\in\zz$, $\overline{L}_0$ intersects $L_i$ in exactly one vertex of $\G$, so that we set this point to be the origin of $L_i$.
		
		Note that the intersection of $\overline{L}_0$ with the whole grid $\G$ is a disjoint union of edges, each joining a couple of paths $L_i,L_{i+1}$, for some $i\in\zz$. Precisely, we write
		
		\[
		\overline{L}_0\cap\G=\bigsqcup_{i\in\zz}b_{2i}^0
		\]
		
		\noindent where, given $i\in\zz$, $b_{2i}^0$ denotes the bridging edge between $L_{2i}$ and $L_{2i+1}$ that belongs to $\overline{L}_0$.
		
		Similarly, for every $k\in\zz$, let $\overline{L}_k$ be the straight line in the plane parallel to $\overline{L}_0$ passing through the vertex of $\vv\in L_0$ corresponding to $x_{L_0}(\vv)=k$, so that
		
		\[
		\overline{L}_k\cap\G=\begin{cases}
		\bigsqcup_{i\in\zz}b_{2i}^k & \text{if $k$ even}\\
		\bigsqcup_{i\in\zz}b_{2i-1}^k & \text{if $k$ odd}
		\end{cases}
		\]
		
		\noindent where again $b_{2i}^k$ (resp. $b_{2i-1}^k$) is the edge of $\G$ joining $L_{2i}$ with $L_{2i+1}$ (resp. $L_{2i-1}$ with $L_{2i}$) that belongs to $\overline{L}_k$.
		
		Moreover, identifying each $b_j^k$ with the interval $[0,1]$ through the coordinate $x_{b_j^k}:b_j^k\to[0,1]$, we stick to the following agreement: if $j\geq0$, then we set $x_{b_j^k}(\vv)=0$ for $\vv=b_j^k\cap L_j$, whereas if $j<0$, then we set $x_{b_j^k}(0)=\vv$ for $\vv=b_j^k\cap L_{j+1}$.
		
		Then, given $\varepsilon>0$, we define (see Figure \ref{fig-neg})
		
		\[
		u_\varepsilon(x):=\begin{cases}
		e^{-\varepsilon(|x|+|i|)} & \text{if $x\in L_i$, for some $i\in\zz$}\\
		e^{-\varepsilon(|x|+|i|+j)} & \text{if $x\in b_j^i$, for some $j,k\in\zz$, $j\geq0$}\\
		e^{-\varepsilon(|x|+|i|+|j+1|)} & \text{if $x\in b_j^i$, for some $j,k\in\zz$, $j<0$}\,.
		\end{cases}
		\]
		
		
		\begin{figure}
			\centering
			\includegraphics[width=0.5\textwidth]{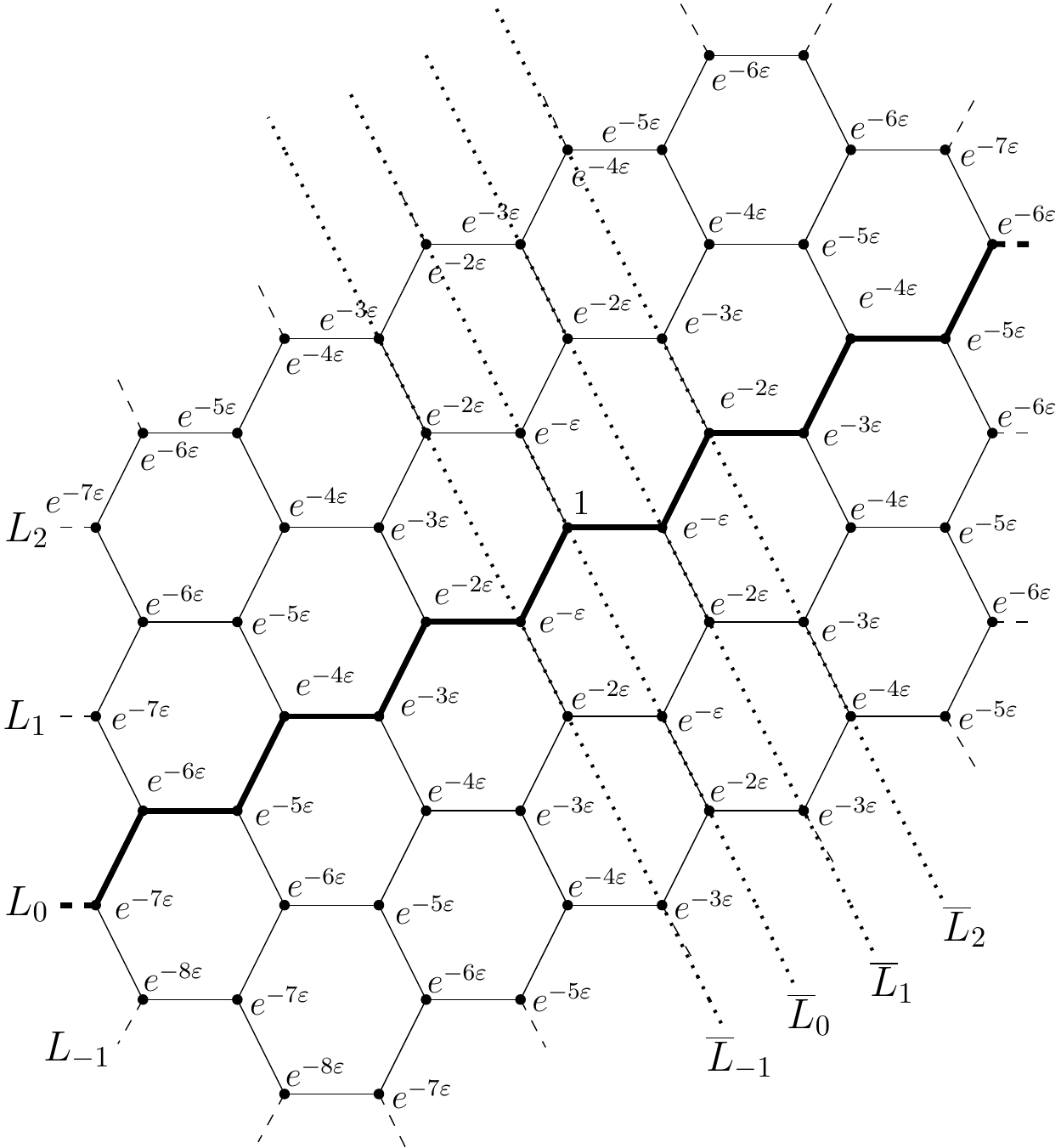}
			\caption{the construction of the function $u$ in the proof of Theorem \ref{THM}, with the straight lines $\overline{L}_i$ and the values of $u$ at the vertices of $\G$.}
			\label{fig-neg}
		\end{figure}
		
		\noindent By construction, $u\in H^1(\G)$ and, given $i\in\zz$,
		
		\[
		\begin{split}
		\int_{L_i}|u_\varepsilon|^p\dx=&\,2\int_0^{+\infty}e^{-p\varepsilon (x+|i|)}\dx=\frac{2e^{-p\varepsilon |i|}}{p\varepsilon}\\
		\int_{\overline{L}_i\cap\G}|u_\varepsilon|^p\dx=&\int_0^{+\infty}e^{-p\varepsilon (x+|i|)}\dx=\frac{e^{-p\varepsilon |i|}}{p\varepsilon}
		\end{split}
		\]
		
		\noindent for every $p\geq2$ and
		
		\[
		\begin{split}
		\int_{L_i}|u_\varepsilon'|^2\dx=&\,2\varepsilon^2\int_0^{+\infty}e^{-2\varepsilon(|x|+|i|)}\dx=\varepsilon e^{-2\varepsilon|i|}\\
		\int_{\overline{L}_i\cap\G}|u_\varepsilon'|^2\dx=&\,\varepsilon^2\int_0^{+\infty}e^{-2\varepsilon (x+|i|)}\dx=\frac{\varepsilon e^{-2\varepsilon |i|}}{2}\,.
		\end{split}
		\]
		
		\noindent Since $\G=\Big(\bigcup_{i\in\mathbb{Z}}L_i\Big)\cup\Big(\bigcup_{i\in\mathbb{Z}}\overline{L}_i\cap\G\Big)$, we get
		
		\[
		\begin{split}
		\int_\G|u_\varepsilon|^p\dx=&\sum_{i \in \mathbb{Z}}\int_{L_i}|u_\varepsilon|^p\dx+\sum_{i \in \mathbb{Z}}\int_{\overline{L}_i\cap\G}|u_\varepsilon|^p\dx=3\Big(\frac{1}{p\varepsilon}+2\sum_{i=1}^{\infty}\frac{e^{-p\varepsilon i}}{p\varepsilon}\Big)=\frac{3(e^{p\varepsilon}+1)}{p\varepsilon(e^{p\varepsilon}-1)}\\
		\int_\G|u_\varepsilon'|^2\dx=&\sum_{i \in \mathbb{Z}}\int_{L_i}|u_\varepsilon'|^2\dx+\sum_{i \in \mathbb{Z}}\int_{\overline{L}_i\cap\G}|u_\varepsilon'|^2\dx=3\Big(\frac{\varepsilon}{2}+\sum_{i=1}^{\infty}\varepsilon e^{-2\varepsilon i}\Big)=\frac{3\varepsilon(e^{2\varepsilon}+1)}{2(e^{2\varepsilon}-1)}\,.
		\end{split}
		\]
		
		\noindent Hence, setting
		
		\[
		k_\varepsilon:=\Big(\,\frac{2\varepsilon(e^{2\varepsilon}-1)}{3(e^{2\varepsilon}+1)}\mu\,\Big)^{1/2}
		\]
		
		\noindent and letting 
		
		\[
		v_\varepsilon(x):=k_\varepsilon u_\varepsilon(x)\qquad \forall\, x\in\G
		\]
		
		\noindent yields at
		
		\[
		\|v_\varepsilon\|_{L^2(\G)}^2=k_\varepsilon^2\int_\G|u_\varepsilon|^2\dx=\mu\,.
		\]
		
		\noindent Therefore, $v_\varepsilon\in\HmuG$ for every $\varepsilon>0$ and, taking advantage of the explicit formula above, as $\varepsilon\to0$
		
		\[
		E(v_\varepsilon,\G)=\frac{1}{2}k_\varepsilon^2\int_\G|u_\varepsilon'|^2\dx-\frac{1}{p}k_\varepsilon^p\int_\G|u_\varepsilon|^p\dx\sim\frac{1}{2}\mu\varepsilon^2-\frac{1}{p}C\mu^{p/2}\varepsilon^{p-2}
		\]
		
		\noindent for some $C>0$ depending only on $p$. Thus, whenever $p<4$ and $\varepsilon$ is small enough, we have
		
		\[
		\elevel(\mu)\leq E(v_\varepsilon,\G)<0
		\]
		
		\noindent and we conclude.
	\end{proof}


\end{document}